\documentclass[11pt]{amsart}

\usepackage{amsmath,amsthm,amsfonts,amssymb,bm,graphicx }
\usepackage{graphicx}
\usepackage{enumitem}
\usepackage{xparse}
\usepackage{tikz}
\usepackage{comment}
\usepackage{bm}

\usepackage%[backref=page]
{hyperref}
\hypersetup{colorlinks=true,citecolor=blue,linkcolor=blue,urlcolor=blue}

\theoremstyle{plain}
\newtheorem{theorem}{Theorem}
\newtheorem*{theorem*}{Theorem}
\newtheorem{lemma}[theorem]{Lemma}

\newtheorem{cor}[theorem]{Corollary}
\newtheorem{prop}[theorem]{Proposition}

\theoremstyle{remark}
\newtheorem*{remark}{Remark}
\newtheorem{example}{Example}

\numberwithin{theorem}{section}

\numberwithin{equation}{section}

\def\N{\mathbb N}
\def\Z{\mathbb Z}
\def\R{\mathbb R}
\def\Q{\mathbb Q}
\def\A{\mathcal A}

\tikzset{
	bullet/.style={circle, fill=#1, draw=none, inner sep=0pt, minimum size=5pt},
	legend/.style={rectangle, fill=#1, draw=none, inner sep=0pt, minimum size=12pt},
}

\colorlet{myviolet}{violet!80}
\colorlet{mygray}{gray!60!blue!40}
\colorlet{mygreen}{black!60!green!100!blue}

\newcommand\DrawBullets[4]{
	\pgfmathsetmacro{\amin}{#2}
	\pgfmathsetmacro{\amax}{#3}
	\foreach \a in {\amin,...,\amax} { \node[bullet=#4] at (0.5*\a,0.5*#1) {}; }
}

\begin{document}

\author{Magdal\'ena Tinkov\'a}

\title{Arithmetics in number systems with cubic base}

\address{Department of Mathematics, FNSPE, Czech Technical University in Prague, Trojanova 13,
120 00 Praha 2, Czech Republic}

\address{Charles University, Faculty of Mathematics and Physics, Department of Algebra,
Sokolovsk\'{a} 83, 18600 Praha 8, Czech Republic}
\email{tinkova.magdalena@gmail.com}

\keywords{positional number systems, greedy expansions, Pisot numbers}

\thanks{The author was supported by Grant Agency of the Czech Technical University in Prague, grant No. SGS14/205/OHK4/3T/14, by Czech Science Foundation (GA\v{C}R), grant 17-04703Y, by the Charles University, project GA UK No. 1298218, by Charles University Research Centre program UNCE/SCI/022, and by the project SVV-2017-260456.}

\begin{abstract}
This paper focuses on greedy expansions, one possible representation of numbers, and on arithmetical operations with them. Performing addition or multiplication some additional digits can appear. We study bounds on the number of such digits assuming the finiteness of the expansion of the considered sum or product, especially for the case of cubic Pisot units.  
\end{abstract}

\setcounter{tocdepth}{2}  \maketitle 

\section{Introduction}

Using a base $\beta$ and a finite set of digits, we can express real or complex numbers in many different ways. One possibility, the so-called greedy expansions, was introduced by A. R\'{e}nyi in \cite{Ren}. These representation are, in lexicographical sense, the largest and conserve the order of real numbers. We can recognize them by comparing with the special representation of number $1$ \cite{Par}.

However, when we add or multiply numbers with a finite greedy expansion, i.e., ended with infinitely many zeros, it is not true that the result of these operations has always a finite greedy expansion. This issue is related to property $(F)$ or to property $(PF)$. As it was proved in \cite{FrouSo}, the only candidates for these properties are Pisot numbers.

Moreover, there were introduced $\beta$-integers, an analogy to rational integers, which have zeros after the fractional point in their greedy expansions. Nevertheless, there can appear even infinitely many nonzero digits after the fractional point in the greedy expansions of a sum or of a product of $\beta$-integers. Therefore, there were introduced bounds $L_{\oplus}(\beta)$ and $L_{\otimes}(\beta)$, maximums of the number of such digits. Note that in the definition of these values, we consider only the cases when our sum or product has a finite greedy expansion. 

In \cite{FrouSo, GuMaPe}, the authors proved that bounds $L_{\oplus}(\beta)$ and $L_{\otimes}(\beta)$ are finite for Pisot numbers. Currently, we know quite precise estimates for quadratic Pisot number, see \cite{Bur, GuMaPe, BaPeTu,lxkva}. In the case of cubic Pisot number, there were found estimates or exact values of these numbers only for few examples \cite{AmMaPe1, AmMaPe2, Ber, Mess}. The main goal of this paper is to extend this set of cubic Pisot numbers.

In this paper, we will restrict our attention to cubic Pisot units with real conjugates. First of all, Section \ref{sec:pre} presents some basic facts about greedy expansions and investigated bounds. Sections \ref{chap++} and \ref{chap:-+} discuss the cases with a sufficiently small positive conjugate, which gives the main result summarized in the following theorem.         
    
\begin{theorem*}
Let $\beta$ be a cubic Pisot unit which has a positive conjugate and does not satisfy property $(PF)$. Then $1\leq L_{\oplus}(\beta)\leq 2$ and $L_{\otimes}(\beta)\leq 4$ except for the dominant root of the polynominal $X^3-2X^2-X+1$, where $2\leq L_{\oplus}(\beta)\leq3$ and $3\leq L_{\otimes}(\beta)\leq 5$. %{\color{red}{except one base?}}
\end{theorem*}

In Section~\ref{chap:+-} we we will look more closely at the remaining case with a positive conjugate, i.e., we will consider $\beta$ satisfying property $(PF)$. However, in this case, we will provide only partial solution of our question. The cubic Pisot units with two negative conjugates are examined in Section~\ref{Chap--}, in which we derive a lower bound on $L_{\oplus}(\beta)$ for wide set of such units and, futhermore, give an upper bound for a few cases applying the idea of Rauzy fractals.     

\section{Preliminaries} \label{sec:pre}

Let $\beta>1$ be a base and $\A=\{0,1,\ldots,\lceil\beta\rceil-1\}$ be an integer alphabet. Throughout this paper, we will work only with positive bases and this type of alphabet. Let $x$ be a real number. Then the $(\beta,\A)$-representation of $x$ is a sequence $(a_i)_{i=-\infty}^N$ of numbers of the alphabet such that
\[
x=\sum_{i=-\infty}^N a_i\beta^i=a_Na_{N-1}\ldots a_1a_0\bullet a_{-1}a_{-2}a_{-3}\ldots
\]
where $N\in\Z$. The number $x$ can have more than one $(\beta,\A)$-representation. We can distinguish these representations using the so-called lexicographic order, denoted by $\succ$. Let $(a_i)_{i=-\infty}^N$ and $(b_i)_{i=-\infty}^N$ be two $(\beta,\A)$-representations of $x$. We say that $(a_i)_{i=-\infty}^N$ is lexicographically greater than $(b_i)_{i=-\infty}^N$, denoted by \[(a_i)_{i=-\infty}^N\succ(b_i)_{i=-\infty}^N,\] if there exists $j\in\Z$ such that $a_i=b_i$ for all $i\in\{j+1,j+2,\ldots,N\}$ and $a_j>b_j$. 

In this paper, we focus on one particular representation, the so-called greedy expansion. We can obtain it using the following algorithm. Let $x\in[0,1)$. Set
\begin{enumerate}
\item $r_0=x$,
\item $x_j=\lfloor \beta r_{j-1}\rfloor$, $r_j=\beta r_{j-1}-\lfloor\beta r_{j-1}\rfloor$.
\end{enumerate}
Then $0\bullet x_1 x_2 x_3\ldots$ is the greedy expansion of $x$ in base $\beta$. We denote this expansion by $\langle x\rangle_{\beta}$. To get the greedy expansion of $x>1$, we firstly divide $x$ by the least power of $\beta$, say $\beta^n$, such that $\frac{x}{\beta^n}\in[0,1)$. Then we apply the algorithm on $\frac{x}{\beta^n}$ and move the fractional point in its greedy expansion to obtain the greedy expansion of $x$. 

In the case of $x<0$, we use a sign to define the greedy expansion of $x$. A number $x$ has a finite greedy expansion if it is ended with infinitely many zeros. We say that a greedy expansion is eventually periodic if it is ended with a periodical repetition of some finite sequence of digits $x_lx_{l-1}\ldots x_{L+1}x_L$. We denote this repetition by $(x_lx_{l-1}\ldots x_{L+1}x_L)^{\omega}$. 

If we set $x=1$ in the previous algorithm, the obtained string $t_1t_2t_3\ldots$ is the so-called R\'{e}nyi expansion of $1$ in base $\beta$. We denote it by $d_{\beta}(1)$. Note that we do not consider the greedy expansion of $1$, which is equal to $1$. We also use the infinite R\'{e}nyi expansion of $1$ defined by
\begin{itemize}
\item $d_{\beta}^{*}(1)=(t_1t_2\ldots t_{m-1}(t_m-1))^{\omega}$ if $d_{\beta}(1)=t_1t_2\ldots t_{m-1}t_m(0)^{\omega}$,
\item $d_{\beta}^{*}(1)=d_{\beta}(1)$ otherwise. 
\end{itemize} 
Note that $\lfloor\beta\rfloor=t_1$. This particular expansion of $1$ helps us to decide if some representation of a number $x$ is also the greedy expansion of this number. In \cite{Par}, the following theorem was proved.

\begin{theorem}[\cite{Par}]
Let $(a_i)_{i=-\infty}^N$ be a $(\beta,\A)$-representation of $x$. Then $(a_i)_{i=-\infty}^N$ is the greedy expansion of $x$ in base $\beta$ if and only if
\[
(a_i)_{i=-\infty}^J\prec t_1t_2t_3\ldots
\]
for all $J\in \Z$.
\end{theorem}

Moreover, in \cite{Bas}, Bassino has found the R\'{e}nyi expansions of $1$ for all the cubic Pisot numbers. We can also cite \cite{Aki3} as a source for $d_{\beta}(1)$ of cubic Pisot units. 

In this paper, we will work with following sets:
\[
\text{fin}(\beta)=\{x\in\R: |x| \text{ has a finite greedy expansion}\}, 
\]
\[
\Z_{\beta}=\{x\in\R: |x| \text{ has a greedy expansion of the form } a_Na_{N-1}\ldots a_0\bullet (0)^{\omega}\}.
\]
Elements of $\Z_{\beta}$ are called $\beta$-integers and are analogous to rational integers. If $\text{fin}(\beta)\cap \R_0^+$ is closed under addition, we say that $\beta$ satisfies property $(PF)$. If the whole set $\text{fin}(\beta)$ is closed under addition, then $\beta$ satisfies property $(F)$. Note that $\beta$ satisfying property $(F)$ has also property $(PF)$. It is known that $\beta$ having property $(PF)$ is necessarily a Pisot number  \cite{FrouSo}. It means that $\beta>1$ is an algebraic integer and the absolute values of all its conjugates are less than $1$. However, it is not true that all the Pisot number satisfy property $(PF)$.

We next define two numbers which we want to examine in this paper. Let
\[
L_{\oplus}(\beta)=\min\{L\in\N_{0}:\text{for all }x,y\in\Z_{\beta},x+y\in\text{fin}(\beta)\Rightarrow \beta^{L}(x+y)\in\Z_{\beta}\},
\]
\[
L_{\otimes}(\beta)=\min\{L\in\N_{0}:\text{for all }x,y\in\Z_{\beta},xy\in\text{fin}(\beta)\Rightarrow \beta^{L}(xy)\in\Z_{\beta}\}.
\]
Since $\text{fin}(\beta)$ is not always closed under addition, i.e., sum or product of two elements of $\text{fin}(\beta)$ can have a greedy expansion with infinitely many nonzero digits, the conditions  $x+y\in\text{fin}(\beta)$ and $xy\in\text{fin}(\beta)$ are crucial in these definitions. If $\beta$ is a Pisot number, then $L_{\oplus}$ and $L_{\otimes}$ are finite \cite{FrouSo,GuMaPe}. Moreover, in the case of Pisot numbers, we also know that there exists a finite set $F$ such that $\Z_{\beta}+\Z_{\beta}\subset \Z_{\beta}+F$ and $\Z_{\beta}\Z_{\beta}\subset\Z_{\beta}+F$.

\begin{figure}%[htp]
\flushleft

\begin{tikzpicture}[scale=0.55]   
\fill[mygray!30] (-0.25,-0.5)--(5.25,5)--(5.25,-5.25)--(4.5,-5.25)--cycle;
\fill[red!30] 
(5.25,0.75)--(1,0.75)--(-0.25,-0.5)--(4.5,-5.25)--(5.25,-5.25)--cycle;
\draw[step=0.5, lightgray, line width = 0.01] (-5.25,-5.25) grid (5.25,5.25);
\draw[->] (-5.25,0) -- (5.5,0) coordinate () node[anchor=north] {$a$};
\draw[->] (0,-5.25) -- (0,5.5) coordinate () node[anchor=east] {$b$};
\draw (0.5,0) node[above] {\scriptsize $\bm{1}$}
      (0,-0.5) node[left] {\scriptsize $\bm{-1}$};
%\node[legend=red] (nnn) at (5.75,3.75) {};
%\node[right=2pt of nnn,red] {Vlastnost $(F)$};
\foreach \b in {2,...,9}{\DrawBullets{\b}{10}{\b+1}{mygray}}
\foreach \b in {0,...,1}{\DrawBullets{\b}{10}{\b+1}{red}}
\foreach \b in {-10,...,-1}{\DrawBullets{\b}{10}{-\b-1}{red}}  
\end{tikzpicture}
\begin{tikzpicture}[scale=0.55]
\fill[mygray!30] (0.75,-0.5)--(5.25,4)--(5.25,-5)--cycle;
\fill[green!30] 
(1,-0.25)--(5.25,4)--(5.25,-0.25)--cycle;      
\draw[step=0.5, lightgray, line width = 0.01] (-5.25,-5.25) grid (5.25,5.25);
\draw[->] (-5.25,0) -- (5.5,0) coordinate () node[anchor=north] {$a$};
\draw[->] (0,-5.25) -- (0,5.5) coordinate () node[anchor=east] {$b$};
\draw	(1,0) node[above] {\scriptsize $\bm{2}$}
        (0,-0.5) node[left] {\scriptsize $\bm{-1}$};
\foreach \b in {0,...,7}{\DrawBullets{\b}{10}{\b+3}{green}}
\foreach \b in {-9,...,-1}{\DrawBullets{\b}{10}{-\b+1}{mygray}}        
\end{tikzpicture}
\flushleft
\begin{tikzpicture}[scale=0.25]
\fill[red]
(-5,0.5)--(-5,-0.5)--(-4,-0.5)--(-4,0.5)--cycle;
\draw (0,0) node[red] {\scriptsize Property $(F)$};  
\end{tikzpicture}
\begin{tikzpicture}[scale=0.25]
\fill[green]
(-5,0.5)--(-5,-0.5)--(-4,-0.5)--(-4,0.5)--cycle;
\draw (0,0) node[green] {\scriptsize  Property $(PF)$};  
\end{tikzpicture}
\caption{Coefficients of the minimal polynomial $X^3-aX^2+bX+c$ of cubic Pisot units for $c=-1$ and $c=1$, distinguished by the properties $(F)$ and $(PF)$.}
\label{Fig:1}
\end{figure}
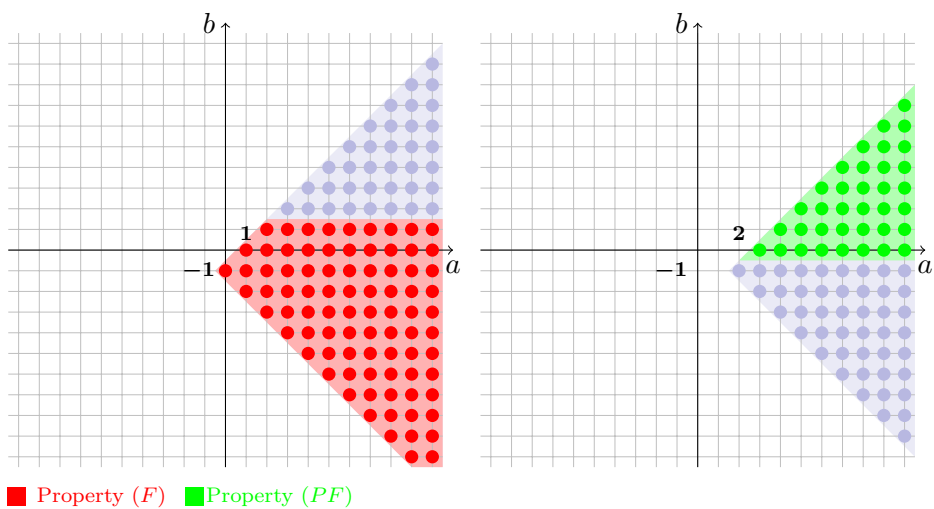

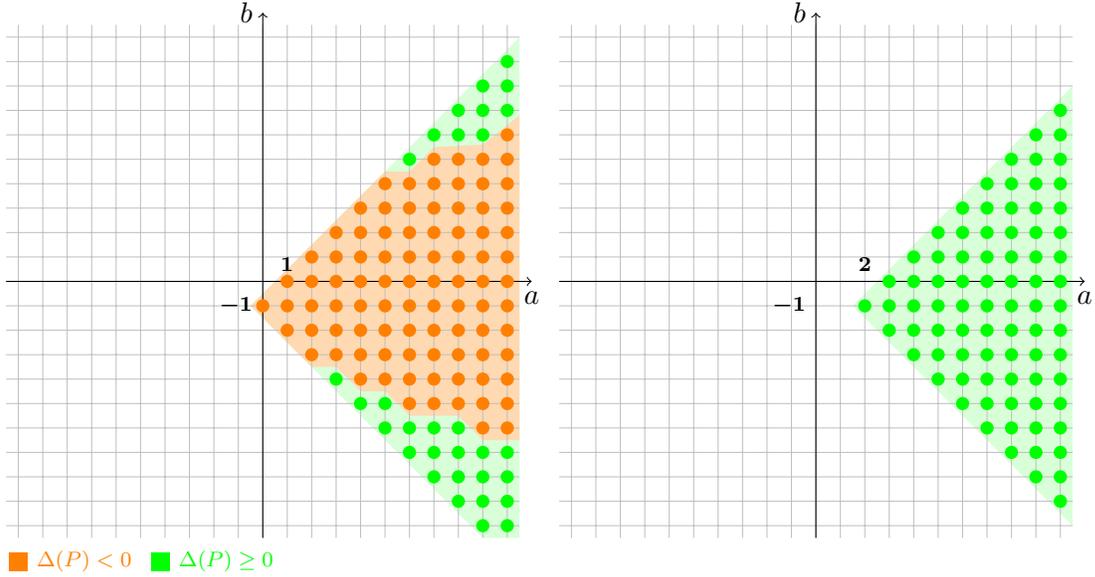
\begin{figure}%[htp]
\centering
\begin{tikzpicture}[scale=0.65]    
\fill[green!15] (-0.25,-0.5)--(5.25,5)--(5.25,-5.25)--(4.5,-5.25)--cycle;
\fill[orange!30]
(5.25,0.5)--(5.25,3.4)--(4.5,2.8)--(3.5,2.75)--(3,2.25)--(2.5,2.25)--(2,1.75)--(-0.25,-0.5)--(1,-1.75)--(1.5,-1.75)--(2,-2.25)--(2.5,-2.25)--(3,-2.75)--(4,-2.75)--(4.5,-3.25)--(5.25,-3.25)--cycle;
\draw[step=0.5, lightgray, line width = 0.01] (-5.25,-5.25) grid (5.25,5.25);
\draw[->] (-5.25,0) -- (5.5,0) coordinate () node[anchor=north] {$a$};
\draw[->] (0,-5.25) -- (0,5.5) coordinate () node[anchor=east] {$b$};
\draw (0.5,0) node[above] {\scriptsize $\bm{1}$}
      (0,-0.5) node[left] {\scriptsize $\bm{-1}$};
\foreach \b in {0,...,1}{\DrawBullets{\b}{10}{\b+1}{orange}}
\foreach \b in {-3,...,-1}{\DrawBullets{\b}{10}{-\b-1}{orange}}
\DrawBullets{2}{10}{3}{orange}
\DrawBullets{3}{10}{4}{orange}
\DrawBullets{4}{10}{5}{orange}
\DrawBullets{5}{10}{7}{orange}
\DrawBullets{6}{10}{10}{orange}
\DrawBullets{9}{10}{10}{green}
\DrawBullets{8}{10}{9}{green}
\DrawBullets{7}{10}{8}{green}
\DrawBullets{6}{9}{7}{green}
\DrawBullets{5}{6}{6}{green}
\DrawBullets{-4}{10}{4}{orange} 
\DrawBullets{-5}{10}{6}{orange}
\DrawBullets{-6}{10}{9}{orange}
\DrawBullets{-4}{3}{3}{green}
\DrawBullets{-5}{5}{4}{green}
\DrawBullets{-6}{8}{5}{green}
\DrawBullets{-7}{10}{6}{green}
\DrawBullets{-8}{10}{7}{green}
\DrawBullets{-9}{10}{8}{green}
\DrawBullets{-10}{10}{9}{green}      
\end{tikzpicture}
\begin{tikzpicture}[scale=0.65]
\fill[green!15] (0.75,-0.5)--(5.25,4)--(5.25,-5)--cycle;     
\draw[step=0.5, lightgray, line width = 0.01] (-5.25,-5.25) grid (5.25,5.25);
\draw[->] (-5.25,0) -- (5.5,0) coordinate () node[anchor=north] {$a$};
\draw[->] (0,-5.25) -- (0,5.5) coordinate () node[anchor=east] {$b$};
\draw	(1,0) node[above] {\scriptsize $\bm{2}$}
        (0,-0.5) node[left] {\scriptsize $\bm{-1}$};
\foreach \b in {-1,...,7}{\DrawBullets{\b}{10}{\b+3}{green}}
\foreach \b in {-9,...,-2}{\DrawBullets{\b}{10}{-\b+1}{green}}        
\end{tikzpicture}
\begin{tikzpicture}[scale=0.25]
\fill[orange]
(-5,0.5)--(-5,-0.5)--(-4,-0.5)--(-4,0.5)--cycle;
\draw (-1,0) node[orange] {\scriptsize $\Delta(P)<0$};  
\end{tikzpicture}
\begin{tikzpicture}[scale=0.25]
\fill[green]
(-5,0.5)--(-5,-0.5)--(-4,-0.5)--(-4,0.5)--cycle;
\draw (-1,0) node[green] {\scriptsize $\Delta(P)\geq0$};  
\end{tikzpicture}
\begin{tikzpicture}[scale=0.25]
\fill[white]
(-5,0.5)--(-5,-0.5)--(-4,-0.5)--(-4,0.5)--cycle;
\draw (0,0) node[white] {\scriptsize mmmmmmmmmmmmmmmmmmmmmmmmmmmmmmmmmmmmmmmmmm};  
\end{tikzpicture}
\caption{Coefficients of the minimal polynomial $X^3-aX^2+bX+c$ of cubic Pisot units for $c=-1$ and $c=1$, distinguished by the discriminant.}
\label{Fig:2}
\end{figure}

Nearly sharp bounds on $L_{\oplus}(\beta)$ have been found for all quadratic Pisot numbers. Our goal in this paper is to determine them for some cubic Pisot unit bases. In \cite{Aki3}, Akiyama characterized all the cubic Pisot numbers, we show the results only for units.

\begin{theorem}[\cite{Aki3}]
A number $\beta$ is a cubic Pisot unit if and only if $\beta$ is the dominant root of the polynomial 
\[
P(X)=X^3-aX^2+bX-1
\]
where $|b+1|<a+1$, or of the polynomial
\[
P(X)=X^3-aX^2+bX+1
\]
where $|b+1|<a-1$.
\end{theorem}

We also know which cubic Pisot units satisfy properties $(PF)$ and $(F)$ \cite{Aki3,Aki2}, see Figure~\ref{Fig:1}. They can be also divided into two parts according to the discriminant $\Delta(P)$, see Figure~\ref{Fig:2}. If $\Delta(P)\geq 0$, the polynomial $P$ has three real roots, otherwise it has one real and two complex roots. In our paper, we focus on the first case.  

If we want to derive a lower bound on $L_{\oplus}(\beta)$ or $L_{\otimes}(\beta)$, one way is to find an example of two $\beta$-integers whose sum has a suitable greedy expansion. On the other hand, we can apply several different methods to find an upper bound. Except for Chapter \ref{Chap--}, we use the so-called $HK$ method or its modification. Let $\beta'$ be a Galois conjugate of $\beta$. Set
\[
H=\sup\{|x'|:x\in\Z_{\beta}\},
\]
\[
K=\sup\{|x'|:x\in\Z_{\beta}\setminus \beta\Z_{\beta}\},
\]
where $'$ is the isomorphism between $\Q(\beta)$ and $\Q(\beta')$. Note that the values of $H$ and $K$ depend on the choice of the conjugate $\beta'$. Thus, in this paper, we always specify which conjugate we consider. If $|\beta'|<1$ and $K>0$, then 
\[
\left(\frac{1}{|\beta'|}\right)^{L_{\oplus}(\beta)}<\frac{2H}{K}
\]
and 
\[
\left(\frac{1}{|\beta'|}\right)^{L_{\otimes}(\beta)}<\frac{H^2}{K},
\]
see \cite{GuMaPe}. We also know that if $|\beta'|<1$, the number $H$ can be estimated as $H\leq\frac{\lfloor\beta\rfloor}{1-|\beta'|}$. Moreover, if $\beta'\in(0,1)$, then $K=1$.

\section{Case with two positive conjugates} \label{chap++}

In the $HK$ method, it is convenient to have a base with a positive conjugate, which simplifies some steps in this estimation. Actually, we know examples of cubic Pisot unit bases which have even two positive conjugate, see \cite{MaTi}. In this case, we work with polynomials of the form
\[
P(X)=X^3-aX^2+bX-1
\]
where $2\leq b<a$ and the dicriminant $\Delta(P)$ is nonnegative. We also know that $a\geq 6$, which can be deduced from the discriminant of cubic polynomials.

These bases do not satisfy either property $(F)$ or property $(PF)$, which implies that a sum of two numbers of $\text{fin}(\beta)$ may have an infinite greedy expansion. In \cite{Bas, Aki3}, the authors have found the R\'{e}nyi expansion of $1$ of the form $d_{\beta}(1)=(a-1)(a-b-1)(a-b)^{\omega}$, thus the alphabet is $\A=\{0,1,\ldots,a-1\}$.

First of all, we will obtain a lower bound on $L_{\oplus}(\beta)$.  

\begin{lemma} \label{lem:low++}
Let $\beta$ be the dominant root of the polynomial
\[
P(X)=X^3-aX^2+bX-1
\]
where $\Delta(P)\geq 0$ and $2\leq b<a$. Then $L_{\oplus}(\beta)\geq 1$.
\end{lemma}

\begin{proof}
In this case, the R\'enyi expansion of $1$ has the form $(a-1)(a-b-1)(a-b)^{\omega}$. Moreover, we know that $a-b-1\geq 0$, $a-b\geq 1$ and $a\geq 6$. Set $\langle x\rangle_{\beta}=(a-1)00b$ and $\langle y\rangle_{\beta}= 2000$; these strings are greedy expansions of $\beta$-integers $x$ and $y$. Now we deduce the greedy expansion of their sum. We heavily use the fact that $\beta$ is a root of $X^3-aX^2+bX-1$. 
\[
\begin{matrix}
 &a-1 & 0 & 0 & b & \bullet & \\
 & 2 & 0 & 0 & 0 & \bullet & \\
 \hline
 & a+1 & 0 & 0 & b & \bullet & \\
 1 & -a & b & -1 &  &  & \\
 \hline
 1 & 1 & b & -1 & b & \bullet & \\
  &  & -1 & a & -b & \bullet & 1 \\
\hline
 1 & 1 & b-1 & a-1 & 0 & \bullet & 1 \\  
\end{matrix}
\]
Thus, since this string is lexicographically smaller than the infinite R\'enyi expansion of $1$, we have $\langle x+y\rangle_{\beta}=11(b-1)(a-1)0\bullet 1$. Therefore $L_{\oplus}(\beta)\geq 1$.  
\end{proof}

In the following lemma, we derive a useful estimate of one of our positive conjugates. As we can see in other parts of this paper, similar procedure is used repeatedly. 

\begin{lemma} \label{lemma:root1}
Let $\beta$ be the dominant root of the polynomial
\[
P(X)=X^3-aX^2+bX-1
\]
where $\Delta(P)\geq 0$ and $2\leq b<a$. Then one of its conjugates is less than or equal to $\frac{1}{\sqrt{a-1}}$.
\end{lemma}

\begin{proof}
Since $d_{\beta}(1)=(a-1)(a-b-1)(a-b)^{\omega}$, we have $\lfloor\beta\rfloor=a-1$. Denote by $\beta'$ and $\beta''$ two positive conjugates of $\beta$. To obtain a contradiction, assume that both these conjugates are greater than $\frac{1}{\sqrt{a-1}}$. It follows that
\[
1=\beta\beta'\beta''>(a-1)\frac{1}{\sqrt{a-1}}\frac{1}{\sqrt{a-1}}=1,
\]
which is impossible.
\end{proof}

Applying the $HK$ method, we obtain an upper bound on $L_{\oplus}(\beta)$. We use the estimate found in the previous lemma, which shows that one of positive conjugates is small for large values of $a$.

\begin{prop} \label{prop:apr+++}
Let $\beta$ be the dominant root of the polynomial
\[
P(X)=X^3-aX^2+bX-1
\]
where $\Delta(P)\geq 0$ and $2\leq b<a$. Then $1\leq L_{\oplus}(\beta) \leq 2$ for sufficiently large $a$.
\end{prop}

\begin{proof}
Lemma \ref{lem:low++} gives us the lower bound on $L_{\oplus}(\beta)$. Let $\beta'$ be the conjugate of $\beta$ such that $\beta'\leq \frac{1}{\sqrt{a-1}}$, whose existence is guaranteed by the previous lemma. Since $\beta'>0$, we have $K=1$ in the $HK$ method. Our next claim is that
\[
H\leq\frac{\lfloor\beta\rfloor}{1-\beta'}\leq\frac{a-1}{1-\frac{1}{2}}=2(a-1).
\]
Thus, we deduce that
\[
\left(\frac{1}{\beta'}\right)^{L_{\oplus}(\beta)}<\frac{2H}{K}\leq 4(a-1).
\]
Moreover, we can find a lower bound on $\left(\frac{1}{\beta'}\right)^{L_{\oplus}(\beta)}$ as
\[
(\sqrt{a-1})^{L_{\oplus}(\beta)}\leq \left(\frac{1}{\beta'}\right)^{L_{\oplus}(\beta)} 
\]
and get
\[
(\sqrt{a-1})^{L_{\oplus}(\beta)}<4(a-1).
\]  
Applying the function $\ln$ and rearranging the inequality we obtain
\[
L_{\oplus}(\beta)<\frac{\ln(4(a-1))}{\ln\sqrt{a-1}}=2\left(1+\frac{\ln4}{\ln(a-1)}\right).
\]
Therefore, we have $L_{\oplus}(\beta)\leq2$ for large $a$.
\end{proof}

\begin{remark}
Replacing $\frac{1}{2}$ by $\frac{1}{3}$ in the previous proof we see that
\[
L_{\oplus}(\beta)<2\left(1+\frac{\ln3}{\ln(a-1)}\right).
\]
Therefore $L_{\oplus}(\beta)\leq 2$ for $a\geq 10$. The cases where $a<10$ are treated in the following example.
\end{remark}

\begin{example} \label{ex:+++}
We have finitely many bases of this type on which we have not decided yet. However, if we use an approximate value of our conjugates in the $HK$ method, we get the estimates given in Table \ref{tab:1}.
\begin{table}[h]
\centering 
\begin{tabular}{|c|c|}
\hline
Polynomial & $L_{\oplus}(\beta)$\\
\hline
$X^3-6X^2+5X-1$ & $\leq 2$\\
\hline
$X^3-7X^2+6X-1$ & $1$\\
\hline
$X^3-8X^2+6X-1$ & $\leq 2$\\
\hline
$X^3-9X^2+6X-1$ & $\leq 2$\\
\hline
\end{tabular}
\quad
\begin{tabular}{|c|c|}
\hline
Polynomial & $L_{\oplus}(\beta)$\\
\hline
$X^3-8X^2+7X-1$ & $1$\\
\hline
$X^3-9X^2+7X-1$ & $1$\\
\hline
$X^3-9X^2+8X-1$ & $1$\\
\hline
\end{tabular}
\caption{$L_{\oplus}(\beta)$ for $a\leq 9$ and two positive conjugates} \label{tab:1}
\end{table}
 
\end{example}

Summaring the results of Proposition \ref{prop:apr+++} and Example \ref{ex:+++}, we obtain the following statement.

\begin{prop} \label{prop:+++}
Let $\beta$ be the dominant root of the polynomial
\[
P(X)=X^3-aX^2+bX-1
\]
where $\Delta(P)\geq 0$ and $2\leq b<a$. Then $1\leq L_{\oplus}(\beta) \leq 2$.
\end{prop}

We can also provide some results about the bound $L_{\otimes}(\beta)$. We will use the same tools as in the case of $L_{\oplus}(\beta)$ including the estimation of one of the conjugates of $\beta$. Finally, we will also have to examine finitely many remaining bases. 

\begin{prop} \label{prop:apr++x}
Let $\beta$ be the dominant root of the polynomial
\[
P(X)=X^3-aX^2+bX-1
\]
where $\Delta(P)\geq 0$ and $2\leq b<a$. Then $L_{\otimes}(\beta) \leq 4$ for sufficiently large $a$.
\end{prop} 

\begin{proof}
For sufficiently large $a$, we can estimate one of the positive conjugates, say $\beta'$, by $\beta'\leq\frac{1}{3}$. Then we get the inequality
\[
L_{\otimes}(\beta)<4\left(1+\frac{\ln\frac{3}{2}}{\ln(a-1)}\right),
\]
which gives our assertion. 
\end{proof}

\begin{remark}
In the previous proof, $\beta'\leq\frac{1}{3}$ if $a\geq 10$. Considering the last inequality in this proof it also follows that $L_{\otimes}(\beta)\leq 4$ for such $a$'s.
\end{remark}

\begin{example} \label{ex:++x}
We have again finitely many cases on which we have not decided yet. Table \ref{tab:2} containing the same polynomials as for addition shows our results.
\begin{table}[h]
\centering 
\begin{tabular}{|c|c|}
\hline
Polynomial & $L_{\otimes}(\beta)$\\
\hline
$X^3-6X^2+5X-1$ & $\leq 3$\\
\hline
$X^3-7X^2+6X-1$ & $\leq 2$\\
\hline
$X^3-8X^2+6X-1$ & $\leq 3$\\
\hline
$X^3-9X^2+6X-1$ & $\leq 3$\\
\hline
\end{tabular}
\quad
\begin{tabular}{|c|c|}
\hline
Polynomial & $L_{\otimes}(\beta)$\\
\hline
$X^3-8X^2+7X-1$ & $\leq 2$\\
\hline
$X^3-9X^2+7X-1$ & $\leq 2$\\
\hline
$X^3-9X^2+8X-1$ & $\leq 2$\\
\hline
\end{tabular}
\caption{$L_{\otimes}(\beta)$ for $a\leq 9$ and two positive conjugates} \label{tab:2}
\end{table} 

\end{example}

We summarize Proposition \ref{prop:apr++x} and Example \ref{ex:++x} in the following statement.

\begin{prop} \label{prop:++x}
Let $\beta$ be the dominant root of the polynomial
\[
P(X)=X^3-aX^2+bX-1
\]
where $\Delta(P)\geq 0$ and $2\leq b<a$. Then $L_{\otimes}(\beta) \leq 4$.
\end{prop}

Proposition \ref{prop:+++} and Proposition \ref{prop:++x} form one part of the main theorem stated in the introduction.

\section{Case with small positive conjugate} \label{chap:-+}

As we have seen in the previous section, it is useful to have a base with a small positive conjugate. If a norm of the given cubic Pisot unit base is equal to $-1$, we have a base with one positive and one negative conjugate. In this section, we will discuss the cases when the positive conjugate is less than the absolute value of its negative conjugate.

Assuming this case we work with the polynomials of the form
\[
P(X)=X^3-aX^2-bX+1
\]
where $0<b<a$. Moreover, as in the previous section, these bases do not satisfy either of property $(F)$ or $(PF)$. The R\'{e}nyi expansion of 1 is equal to $d_{\beta}(1)=a[(b-1)(a-1)]^{\omega}$ and we consider the alphabet $\A=\{0,1,\ldots,a\}$. 

We will start with a lower bound on $L_{\oplus}(\beta)$.

\begin{lemma} \label{lem:low-+1}
Let $\beta$ be the dominant root of the polynomial
\[
P(X)=X^3-aX^2-bX+1
\]
where $0<b<a$. Then $L_{\oplus}(\beta)\geq 1$.
\end{lemma}

\begin{proof}
In this case, $d_{\beta}(1)=a[(b-1)(a-1)]^{\omega}$ and $a\geq2$. If we consider two $\beta$-integer with greedy expansions $\langle x\rangle_{\beta}=a0$ and $\langle y\rangle_{\beta}=1b$, their sum has the greedy expansion of the form $\langle x+y\rangle_{\beta}=110\bullet1$. Thus $L_{\oplus}(\beta)\geq 1$.  
\end{proof}

In some cases, we can improve this lower bound. 

\begin{lemma} \label{lem:low-+2}
Let $\beta$ be the dominant root of the polynomial
\[
P(X)=X^3-aX^2-bX+1
\]
where $0<b<a$. If
\begin{enumerate}
\item $a\geq 2$ and $b=1$ or
\item $a\geq 4$ and $b=2$,
\end{enumerate}
then $L_{\oplus}(\beta)\geq 2$.
\end{lemma}

\begin{proof}
Set $\langle x\rangle_{\beta}=(a-1)b(a-1)a$ and $\langle y\rangle_{\beta}=(a-1)0(a-1)a$. One representation of their sum in base $\beta$ has the form $1(a-2)20(a-2b)\bullet(2-b)1$. This representation is the greedy expansion if all its digits are nonnegative and sufficiently small, which is satisfied for the cases given by the statement of the lemma. We note that it is also true for $a=2$ and $b=1$, since $\langle x+y\rangle_{\beta}=10200\bullet 11$ in this case and, at the same time, $d_{\beta}(1)=2(01)^{\omega}$.   
\end{proof}

Note that in the previous lemma, we omit the case where $b=2$ and $a=3$. As we can see below in this paper, this base has different properties than other bases with a small value of the coefficient $b$. 

We next find an estimate of the positive conjugate of $\beta$, which is similar to the result obtained in Section \ref{chap++}.  

\begin{lemma}
Let $\beta$ be the dominant root of the polynomial
\[
P(X)=X^3-aX^2-bX+1
\]
where $0<b<a$. Then the positive conjugate of $\beta$ is less than or equal to $\frac{1}{\sqrt{a}}$.
\end{lemma}

\begin{proof}
This follows by the same method as in Lemma \ref{lemma:root1}. In this case, the positive conjugate of $\beta'$ is less than the absolute value of the negative one. We have
\[
0>-b=\beta\beta'+\beta\beta''+\beta'\beta''=\beta(\beta'+\beta'')+\beta'\beta''.
\]
Since $-1<\beta'\beta''<0$ and $\beta>0$, the conditition $\beta'+\beta''<0$ is the only way how to obtain a negative integer in this sum. The rest of the proof follows from the fact that $|\beta\beta'\beta''|=1$.
\end{proof}

The following proposition shows that an upper bound on $L_{\oplus}(\beta)$ for large values of $a$ is again constant and significantly small. 

\begin{prop}
Let $\beta$ be the dominant root of the polynomial
\[
P(X)=X^3-aX^2-bX+1
\]
where $0<b<a$. Then $1\leq L_{\oplus}(\beta) \leq 2$ for sufficiently large $a$. 
\end{prop}

\begin{proof}
The proof of this proposition is similar to the proof of Proposition \ref{prop:apr+++}. In this case, we obtain the inequality
\[
L_{\oplus}(\beta)<2\left(1+\frac{\ln3}{\ln a}\right),
\]
thus $L_{\oplus}(\beta)\leq 2$ is satisfied for $a\geq 9$.  
\end{proof}

\begin{example} \label{ex:-+l+}
Our next goal is to find an upper bound on $L_{\oplus}(\beta)$ for $a<9$. The Table \ref{tab:3} contains our results. We have used the approximate values of positive conjugates and, in some cases, we have improved our estimates of $H$ in the $HK$ method. 
\begin{table}[h]
\centering 
\begin{tabular}{|c|c|}
\hline
Polynomial & $L_{\oplus}(\beta)$\\
\hline
$X^3-2X^2-X+1$ & $\leq 3$\\
\hline
$X^3-3X^2-X+1$ & $2$\\
\hline
$X^3-4X^2-X+1$ & $ 2$\\
\hline
$X^3-5X^2-X+1$ & $ 2$\\
\hline
$X^3-6X^2-X+1$ & $ 2$\\
\hline
$X^3-7X^2-X+1$ & $ 2$\\
\hline
$X^3-8X^2-X+1$ & $ 2$\\
\hline
$X^3-3X^2-2X+1$ & $ 1$\\
\hline
$X^3-4X^2-2X+1$ & $ 2$\\
\hline
$X^3-5X^2-2X+1$ & $ 2$\\
\hline
$X^3-6X^2-2X+1$ & $ 2$\\
\hline
$X^3-7X^2-2X+1$ & $ 2$\\
\hline
$X^3-8X^2-2X+1$ & $ 2$\\
\hline
$X^3-4X^2-3X+1$ & $ 1$\\
\hline
\end{tabular}
\quad
\begin{tabular}{|c|c|}
\hline
Polynomial & $L_{\oplus}(\beta)$\\
\hline
$X^3-5X^2-3X+1$ & $1$\\
\hline
$X^3-6X^2-3X+1$ & $1$\\
\hline
$X^3-7X^2-3X+1$ & $ 1$\\
\hline
$X^3-8X^2-3X+1$ & $ 1$\\
\hline
$X^3-5X^2-4X+1$ & $ 1$\\
\hline
$X^3-6X^2-4X+1$ & $ 1$\\
\hline
$X^3-7X^2-4X+1$ & $ 1$\\
\hline
$X^3-8X^2-4X+1$ & $ 1$\\
\hline
$X^3-6X^2-5X+1$ & $ 1$\\
\hline
$X^3-7X^2-5X+1$ & $ 1$\\
\hline
$X^3-8X^2-5X+1$ & $ 1$\\
\hline
$X^3-7X^2-6X+1$ & $ 1$\\
\hline
$X^3-8X^2-6X+1$ & $ 1$\\
\hline
$X^3-8X^2-7X+1$ & $ 1$\\
\hline
\end{tabular}
\caption{$L_{\oplus}(\beta)$ for $a\leq 8$ and small positive conjugate} \label{tab:3}
\end{table}
\end{example}

Combining the results for large and small values of coefficient $a$ gives the following statement.

\begin{prop}
Let $\beta$ be the dominant root of the polynomial
\[
P(X)=X^3-aX^2-bX+1
\]
where $0<b<a$. Then $1\leq L_{\oplus}(\beta) \leq 2$ except for the case $a=2$ and $b=1$, where $2\leq L_{\oplus}(\beta)\leq3$.
\end{prop}

\begin{remark}
The dominant root of the polynomial $X^3-2X^2-X+1$ was in detail investigated in \cite{FroGaKre}. The authors found the corresponding set $F$ such that $\Z_{\beta}+\Z_{\beta}\subset \Z_{\beta}+F$. The elements contained in $F$ suggest that there may exist a sum of $\beta$-integers such that its greedy expansion has three additional digits after the fractional point. 
\end{remark}

Moreover, we can provide a stronger result for small values of $b$. 

\begin{cor}
Let $\beta$ be the dominant root of the polynomial
\[
P(X)=X^3-aX^2-bX+1
\]
where $0<b<a$. If
\begin{enumerate}
\item $a\geq 3$ and $b=1$ or %{\color{red}{ except one base??}} or
\item $a\geq 4$ and $b=2$,
\end{enumerate}
then $L_{\oplus}(\beta)=2$.
\end{cor}

As we can see in Example \ref{ex:-+l+}, $L_{\oplus}(\beta)=1$ for significantly many cases. This observation leads us to the following proposition.

\begin{prop} \label{prop:-+z+}
Let $z\in \N$. Then for large $a\in\N$ and $\beta>1$ the dominant root of the polynomial
\[
X^3-aX^2-(a-z)X+1
\]
we have $L_{\oplus}(\beta)=1$.
\end{prop}

\begin{proof}
Let $\beta'$ be the positive root of $\beta$ and $\beta''$ be the negative one. We know that $z-a=\beta\beta'+\beta\beta''+\beta'\beta''$. Moreover, our roots satisfy that $\lfloor\beta\rfloor=a$, $\beta'<1$ and $|\beta''|<1$. Therefore
\[
z-a=\beta\beta'+\beta\beta''+\beta'\beta''>a\beta'-(a+1)-1,
\]
which implies that $\beta'<\frac{z+2}{a}$. We have $\beta'\leq \frac{1}{2}$ for large values of $a$. Applying the $HK$ method in the same manner as in previous cases we obtain the inequality
\[
L_{\oplus}(\beta)<1+\frac{\ln(4(z+2))}{\ln a-\ln(z+2)},
\]
from which, combining with the lower bound, follows the statement of the proposition. 
\end{proof}

\begin{remark}
We can easily determine the required condition on $a$ in the previous proposition. We need 
\[
\frac{\ln(4(z+2))}{\ln a-\ln(z+2)}\leq 1
\]
to be satisfied. If $a>z+2$, our inequality is equivalent to
\[
4(z+2)\leq\frac{a}{z+2}. 
\]
Thus $L_{\oplus}(\beta)=1$ for $a\geq 4(z+2)^2$. 
\end{remark}

\begin{example}
We now show the application of the previous proposition on the case when $z=1$. We have $4(z+2)^2=36$, moreover, we know that $L_{\oplus}(\beta)=1$ for $a\in\{3,4,\ldots,8\}$, which was shown in Table \ref{tab:3}. We also checked the possible values of $L_{\oplus}$ for $a\in\{9,\ldots,35\}$ and we can claim that $L_{\oplus}=1$ for all these cases. 
\end{example}

Considering this example we can state the following corollary. 

\begin{cor}
Let $\beta$ be the dominant root of the polynomial
\[
P(X)=X^3-aX^2-(a-1)X+1
\]
where $a\geq3$. Then $L_{\oplus}(\beta)=1$. 
\end{cor}

Our next aim is to examine the number $L_{\otimes}(\beta)$ for this type of bases. We will start with a lower bound.

\begin{lemma}
Let $\beta$ be the dominant root of the polynomial
\[
P(X)=X^3-aX^2-bX+1
\]
where $0<b<a$. Then $L_{\otimes}(\beta) \geq 2$.
\end{lemma}

\begin{proof}
Set $\langle x\rangle_{\beta}=(a-1)(b+1)$ and $\langle y\rangle_{\beta}=a$. Then $\langle x\times y\rangle_{\beta}=(a-1)1b\bullet (a-b-1)1$ is the greedy expansion of the product $x\times y$ in base $\beta$ since $0<b<a$. Hence $L_{\otimes}(\beta) \geq 2$.  
\end{proof}

We will improve this lower bound for the case when $b=1$.

\begin{lemma}
Let $\beta$ be the dominant root of the polynomial
\[
P(X)=X^3-aX^2-X+1
\]
where $a\geq 2$. Then $L_{\otimes}(\beta) \geq 3$. 
\end{lemma}

\begin{proof}
In this case, we have $d_{\beta}(1)=a[0(a-1)]^{\omega}$. Put $\langle x\rangle_{\beta}=(a-1)0(a-1)a$. The product $x\times x$ can be expressed as $(a-2)20(a-4)31(a-4)2\bullet 0(a-2)1$. This representation is also the greedy expansion of $x\times x$ if $a\geq 4$. If $a=2$, we have $\langle x\times x\rangle_{\beta}=1120101\bullet001$. If $a=3$, we obtain $\langle x\times x\rangle_{\beta}=12000002\bullet011$. In all the considered cases, we have exactly three additional digits after the fractional point, thus $L_{\otimes}(\beta) \geq 3$.
\end{proof}

Using the small positive conjugate of $\beta$ we can find an upper bound on $L_{\otimes}(\beta)$ for large values of the coefficient $a$. 

\begin{prop}
Let $\beta$ be the dominant root of the polynomial
\[
P(X)=X^3-aX^2-bX+1
\]
where $0<b<a$. Then $2\leq L_{\otimes}(\beta) \leq 4$ for sufficiently large $a$. 
\end{prop}

\begin{example}
Using the upper bound $\frac{1}{3}$ on the positive conjugate we get $2\leq L_{\otimes}(\beta) \leq 4$ for $a\geq 9$. The Table \ref{tab:4} shows the results for the remaining cases. As we can see, there is a stripe with higher upper bounds on $L_{\otimes}(\beta)$, which we have not supported by corresponding lower bounds.    

\begin{table}[h]
\centering 
\begin{tabular}{|c|c|}
\hline
Polynomial & $L_{\otimes}(\beta)$\\
\hline
$X^3-2X^2-X+1$ & $\leq 5$\\
\hline
$X^3-3X^2-X+1$ & $\leq 4$\\
\hline
$X^3-4X^2-X+1$ & $ \leq 4$\\
\hline
$X^3-5X^2-X+1$ & $ \leq 4$\\
\hline
$X^3-6X^2-X+1$ & $ \leq 4$\\
\hline
$X^3-7X^2-X+1$ & $ \leq 4$\\
\hline
$X^3-8X^2-X+1$ & $ \leq 4$\\
\hline
$X^3-3X^2-2X+1$ & $ 2$\\
\hline
$X^3-4X^2-2X+1$ & $ \leq 3$\\
\hline
$X^3-5X^2-2X+1$ & $ \leq 3$\\
\hline
$X^3-6X^2-2X+1$ & $ \leq 3$\\
\hline
$X^3-7X^2-2X+1$ & $ \leq 3$\\
\hline
$X^3-8X^2-2X+1$ & $ \leq 3$\\
\hline
$X^3-4X^2-3X+1$ & $ 2$\\
\hline
\end{tabular}
\quad
\begin{tabular}{|c|c|}
\hline
Polynomial & $L_{\otimes}(\beta)$\\
\hline
$X^3-5X^2-3X+1$ & $2$\\
\hline
$X^3-6X^2-3X+1$ & $2$\\
\hline
$X^3-7X^2-3X+1$ & $ 2$\\
\hline
$X^3-8X^2-3X+1$ & $ \leq 3$\\
\hline
$X^3-5X^2-4X+1$ & $ 2$\\
\hline
$X^3-6X^2-4X+1$ & $ 2$\\
\hline
$X^3-7X^2-4X+1$ & $ 2$\\
\hline
$X^3-8X^2-4X+1$ & $ 2$\\
\hline
$X^3-6X^2-5X+1$ & $ 2$\\
\hline
$X^3-7X^2-5X+1$ & $ 2$\\
\hline
$X^3-8X^2-5X+1$ & $ 2$\\
\hline
$X^3-7X^2-6X+1$ & $ 2$\\
\hline
$X^3-8X^2-6X+1$ & $ 2$\\
\hline
$X^3-8X^2-7X+1$ & $ 2$\\
\hline
\end{tabular}
\caption{$L_{\otimes}(\beta)$ for $a\leq 8$ and small positive conjugate}\label{tab:4}
\end{table}
\end{example}

We summarize our results in the following statement.

\begin{prop}
Let $\beta$ be the dominant root of the polynomial
\[
P(X)=X^3-aX^2-bX+1
\]
where $0<b<a$. Then $2\leq L_{\otimes}(\beta) \leq 4$ except for the case $a=2$ and $b=1$, where $3\leq L_{\otimes}(\beta)\leq5$. %\color{red}{ except one base??}.
\end{prop}

Moreover, we can conclude the following proposition, which is similar to the result for $L_{\oplus}(\beta)$. 

\begin{prop} 
Let $z\in \N$. Then for large $a\in\N$ and $\beta>1$ the dominant root of the polynomial
\[
X^3-aX^2-(a-z)X+1
\]
we have $L_{\otimes}(\beta)=2$.
\end{prop}

\begin{proof}
We use the same procedure as in Proposition \ref{prop:-+z+}. In this case, the statement is satisfied for $a\geq4(z+2)^3$.
\end{proof}

\begin{example}
Let us develop our study of the case $z=1$. In this case, $4(z+2)^3=108$. Hence $L_{\otimes}(\beta)=2$ for $a\geq 108$. Moreover, $L_{\otimes}(\beta)=2$ for $a\in\{3,4,\ldots,8\}$, see Table \ref{tab:4}. We also checked $L_{\otimes}(\beta)$ for $a$ between $8$ and $108$ and we can claim that $L_{\otimes}(\beta)=2$ for all these values of $a$.  
\end{example}

Summarizing it we can state the following corollary. 

\begin{cor}
Let $\beta$ be the dominant root of the polynomial
\[
P(X)=X^3-aX^2-(a-1)X+1
\]
where $a\geq3$. Then $L_{\otimes}(\beta)=2$. 
\end{cor}

If we combine our results derived in Sections \ref{chap++} and \ref{chap:-+}, we obtain the main theorem stated in the introduction.

\section{Case with large positive conjugate} \label{chap:+-}

In this section, we discuss the remaining case of bases with the norm $-1$. We consider the minimal polynomials of the form
\[
P(X)=X^3-aX^2+bX+1
\]
where $0\leq b\leq a-3$. In contrast to the previous cases, these bases satisfy property $(PF)$, i.e., a sum of two nonnegative numbers with a finite greedy expansion has also a finite greedy expansion. This fact, in some sense, complicates the estimation of $L_{\oplus}(\beta)$. The nonnegative coefficient $b$ gives us that the positive conjugate of $\beta$ is larger than the absolute value of the negative conjugate. Furthermore, $d_{\beta}(1)=(a-1)(a-b-1)(a-b-2)^{\omega}$ and $\A=\{0,1,\ldots,a-1\}$.

We proceed with a lower bound on $L_{\oplus}(\beta)$.    

\begin{lemma}
Let $\beta$ be the dominant root of the polynomial
\[
P(X)=X^3-aX^2+bX+1
\]
where $0\leq b\leq a-3$. Then $L_{\oplus}(\beta)\geq 2$. 
\end{lemma}

\begin{proof}
In this case, we have $d_{\beta}(1)=(a-1)(a-b-1)(a-b-2)^{\omega}$. Let $\langle x\rangle_{\beta}=a-1$ and $\langle y\rangle_{\beta}=1$. Then $\langle x+y\rangle_{\beta}=10\bullet b1$. Obviously, this representation is the greedy expansion of $x+y$. 
\end{proof}

In the previous sections, we have derived a useful estimate of the positive conjugate of $\beta$. In this case, we cannot provide similar results. However, we will show how to find another upper bound on the positive conjugate, which we will use in the following application of the $HK$ method.   

\begin{lemma}
Let $\beta$ be the dominant root of the polynomial
\[
P(X)=X^3-aX^2+bX+1
\]
where $0\leq b\leq a-3$. Then the positive conjugate of $\beta$ is less than or equal to
\[
\frac{b+\sqrt{a-1}+\frac{2}{\sqrt{a-1}}}{a-1}.
\]
\end{lemma}

\begin{proof}
Let $\beta'$ be the positive conjugate of $\beta$ and $\beta''$ be the negative one. In the same manner as in previous sections we can prove that $|\beta''|\leq \frac{1}{\sqrt{a-1}}$. Futhermore, we know that $\lfloor\beta\rfloor=a-1$ and $\lceil\beta\rceil=a$. Therefore
\[
b=\beta\beta'+\beta\beta''+\beta'\beta''\geq (a-1)\beta'-a\frac{1}{\sqrt{a-1}}-\frac{1}{\sqrt{a-1}}=(a-1)\beta'-\frac{a-1}{\sqrt{a-1}}-
\] 
\[
-\frac{2}{\sqrt{a-1}}=
(a-1)\beta'-\sqrt{a-1}-\frac{2}{\sqrt{a-1}}.
\]
Rearranging this inequality we obtain the statement of the proposition.   
\end{proof}

If we fix $b$, then the positive conjugate of $\beta$ is small for large values of $a$. Considering this fact we can state the following proposition.

\begin{prop}
Let $\beta$ be the dominant root of the polynomial
\[
P(X)=X^3-aX^2+bX+1
\]
where $0\leq b\leq a-3$. Then for every $b\in\N_{0}$ there exists $a_0\in\N$ such that for every $a\geq a_0$ we have $L_{\oplus}(\beta)=2$.
\end{prop}

\begin{proof}
As before, we use the $HK$ method applied on the positive conjugate $\beta'$ of our base $\beta$. From the previous lemma, we have the upper bound on $\beta'$, i.e.,
\[
\beta'\leq \frac{b+\sqrt{a-1}+\frac{2}{\sqrt{a-1}}}{a-1}. 
\]
Let $b$ be fixed. Then $\beta'\leq \frac{1}{2}$ for large $a$. We can also assume that our bound has similar properties as the fuction $\frac{1}{\sqrt{a-1}}$ for large $a$, so the conjugate $\beta'$ can be estimated by $\beta'\leq\frac{c_b}{\sqrt{a-1}}$ where $c_b$ is a constant depending on the coefficient $b$. 

Applying the $HK$ method we obtain 
\[
L_{\oplus}(\beta)<2\left(1+\frac{\ln(4c_b^2)}{\ln(a-1)-\ln c_b^2}\right),
\]
which completes the proof.
\end{proof}

In the same manner we can derive some assertions about the number $L_{\otimes}(\beta)$. Let us start with a lower bound.

\begin{lemma}
Let $\beta$ be the dominant root of the polynomial
\[
P(X)=X^3-aX^2+bX+1
\]
where $0\leq b\leq a-3$. Then for every $b\in\N\setminus\{1\}$ there exists $a_0\in\N$ such that for every $a\geq a_0$ we have $L_{\otimes}(\beta)\geq4$. 
\end{lemma}

\begin{proof}
Set $\langle x\rangle_{\beta}=a-1$ and consider the greedy expansion of $x\times x$. One representation of this product is equal to $(a-2)b\bullet (a-2b+1)(b^2-b-2)(2b-1)1$. If $b>1$ and $a$ is large enough, then we have obtained the greedy expansion of $x\times x$.
\end{proof}

Note that we excluded the case when $b\in\{0,1\}$. We will solve it separately in the following lemma.

\begin{lemma}
Let $\beta$ be the dominant root of the polynomial
\[
P(X)=X^3-aX^2+bX+1
\]
where $0\leq b\leq a-3$ and $b\in\{0,1\}$. Then $L_{\otimes}(\beta)\geq 4$.
\end{lemma} 

\begin{proof}
Let $b=0$. Then the R\'{e}nyi expansion of 1 has the form $d_{\beta}(1)=(a-1)(a-1)(a-2)^{\omega}$. Put $\langle x\rangle_{\beta}=(a-2)(a-1)(a-1)$. The greedy expansion of $x\times x$ is equal to 
\[
(a-2)0(a-1)1(a-2)(a-3)\bullet (a-1)011.
\] 
Note that $a-1\geq 2$.

In the second case, i.e., $b=1$, we have $d_{\beta}(1)=(a-1)(a-2)(a-3)^{\omega}$. Put $\langle x\rangle_{\beta}=(a-1)(a-2)(a-3)$. The greedy expansion of $x\times x$ is equal to
\[
(a-1)(a-2)(a-4)210\bullet(a-4)131.
\]
Note that $a\geq 4$ in this case.   
\end{proof}

An upper bound follows. 

\begin{prop}
Let $\beta$ be the dominant root of the polynomial
\[
P(X)=X^3-aX^2+bX+1
\]
where $0\leq b\leq a-3$. Then for every $b\in\N_{0}$ there exists $a_0\in\N$ such that for every $a\geq a_0$ we have $L_{\otimes}(\beta)=4$.
\end{prop}

\begin{proof}
From the previous lemmas we have the lower bound on $L_{\otimes}(\beta)$. Using the estimate of the positive conjugate we obtain the inequality
\[
L_{\otimes}(\beta)<4\left(1+\frac{\ln(2c_b^2)}{\ln(a-1)-\ln(c_b^2)}\right),
\]
which gives our claim. 
\end{proof}

\section{Case with two negative conjugates} \label{Chap--}

Last case which we examine in this paper is a base with two negative conjugates. As we can see in \cite{MaTi}, these bases are the dominant roots of the polynomials
\[
P(X)=X^3-aX^2-bX-1
\]  
where $b<a+2$ and $\Delta(P)\geq0$. As it was proved in \cite{Aki3}, these bases satisfy property $(F)$, therefore, a sum of two $\beta$-integers has a finite greedy expansion. This time, we work with two cases which are distinguished by their corresponding R\'{e}nyi expansions of $1$, i.e., $d_{\beta}(1)=(a+1)00a1$ if $b=a+1$ and $d_{\beta}(1)=ab1$ otherwise. In this paper, we focus only on $L_{\oplus}(\beta)$ for the second case, for which we can rewrite corresponding minimal polynomials as
\[
P(X)=X^3-aX^2-(a-z)X-1
\]
where $z\in\N_{0}$. We consider the alphabet $\A=\{0,1,\ldots,a\}$. 

We start with a useful lemma, which will help us to find the greedy expansion of the sum of two concrete $\beta$-integers.

\begin{lemma} \label{lemma:prelow--}
Let $\beta$ be the dominant root of the polynomial 
\[
P(X)=X^3-aX^2-(a-z)X-1
\]
where $a>1$ and $z\in\N_{0}$ is such that $a-z>0$. Let $\langle x_k\rangle_{\beta}=a0a0\ldots a0a$ be the greedy expansion of $\beta$-integer $x_k$ where $k$ is the number of $a$'s in this greedy expansion. Then one of the representations of $x_k+x_k$ in $\Z$ has the form
{\small
\[
1(a-1)(z+2)(a-z-3)(2(z+2))(a-2z-5)\ldots ((k-1)(z+2))(a-(k-1)z-(k-1))\bullet
\]
\[
\bullet (-ka+kz+(k-1))(-k).
\]}
\end{lemma}   

\begin{proof}
The proof is by induction on $k$. If $k=1$, then $\langle x_1\rangle_{\beta}=a$ and we can get the required form in the following proces.
\[
\begin{matrix}
& a & \bullet & & &\\
& a & \bullet & & &\\
\hline
& 2a & \bullet & & &\\
1 & -a & \bullet & -a+z & -1\\
\hline
1 & a & \bullet & -a+z & -1\\
\end{matrix}
\]
Assume the formula holds for $k$; we will prove it for $k+1$. To get the similar representation for $k+1$, let move the fractional point in this representation for $k$ a add $2a$ to the new last digit, i.e., the digit on the position 0. By this procedure, we obtain some representation of $x_{k+1}+x_{k+1}$, which does not have the form given in the statement of our lemma. By repeated adding of zero, we can derive it.
{\small
\[
\begin{matrix}
\ldots & a-(k-1)z-(k-1) & -ka+kz+(k-1) & -k & \bullet\\
       &                &              & 2a & \bullet \\
\hline
\ldots & a-(k-1)z-(k-1) & -ka+kz+(k-1) & 2a-k & \bullet\\ 
       &   -k           & ka           & ka-kz & \bullet & k\\
\hline
\ldots & a-(k-1)z-(2k-1) & kz+(k-1) & (2+k)a-kz-k & \bullet& k\\
\end{matrix}
\]}
We next modify only the end of the previous representation.
{\small
\[
\begin{matrix}
\ldots & kz+(k-1) & (2+k)a-kz-k & \bullet& k\\
                 &  k+1     & -(k+1)a& \bullet & -(k+1)a+(k+1)z & -(k+1)\\
\hline
\ldots & k(z+2) & a-kz-k & \bullet& -(k+1)a+(k+1)z+k & -(k+1)\\           
\end{matrix}
\]}
As we can see, we have deduced the desired form. 

\end{proof}

We now proceed with modification of the representation from the previous lemma. Under certain conditions, the following representation is also the greedy expansion of $x_k+x_k$ in base $\beta$. 

\begin{lemma} \label{lemma:low--1}
Let $\beta$ be the dominant root of the polynomial 
\[
P(X)=X^3-aX^2-(a-z)X-1
\]
where $a>1$ and $z\in\N_{0}$ is such that $a-z>0$. Let $\langle x_k\rangle_{\beta}=a0a0\ldots a0a$ be the greedy expansion of $\beta$-integer $x_k$ where $k$ is the number of $a$'s in this greedy expansion. Then one of the representations of $x_k+x_k$ in $\Z$ has the form
{\small
\[
1(a-1)(z+2)(a-z-3)(2(z+2))(a-2z-5)\ldots ((k-1)(z+2))(a-(k-1)z-(2k-1))\bullet
\]
\[
\bullet (kz+2(k-1))(a-kz-(2k-1))((k-1)(z+2))(a-(k-1)z-(2k-3))\ldots.
\]
\[
\ldots (a-2z-3)(z+2)(a-z-1)1.
\]}
\end{lemma}

\begin{proof}
As previously, the proof is by induction on $k$. For $k=1$ we have
\[
\begin{matrix}
1 & a & \bullet & -a+z & -1\\
  & -1& \bullet & a    & a-z & 1\\
\hline
1 & a-1 & \bullet & z & a-z-1 & 1    
\end{matrix}
\] 
where we use the beginning of Lemma \ref{lemma:prelow--}. Suppose that $x_k+x_k$ can be represented in this way and consider the representation of $x_k+x_k$ given by Lemma \ref{lemma:prelow--}. If we conveniently rearrange this representation, then its end has the following form.
{\tiny
\[
\begin{matrix}
\ldots & a-(k-1)z-(k-1) & \bullet & -ka+kz+(k-1) & -k\\
 &    -k           & \bullet & ka & ka-kz & k\\
\hline
\ldots & a-(k-1)z-(2k-1) & \bullet & kz+(k-1) & ka-kz-k & k\\
 &                 &     & k-1 & -(k-1)a & -(k-1)a+(k-1)z & -(k-1)\\                 
\hline
\ldots & a-(k-1)z-(2k-1) & \bullet & kz+2(k-1) & a-kz-k & -(k-1)a-(k-1)z+k & -(k-1) \\             
\end{matrix}
\]}

Let us consider the representation of $x_{k+1}+x_{k+1}$ from the previous lemma. We can also change it to the following form.
{\scriptsize
\[
\begin{matrix}
\ldots & a-kz-k & \bullet & -(k+1)a+(k+1)z+k & -(k+1)\\
 & -(k+1) & \bullet & (k+1)a & (k+1)a-(k+1)z & k+1\\
\hline
\ldots & a-kz-(2k+1) & \bullet & (k+1)z+k & (k+1)a-(k+1)z-(k+1) & k+1\\
\end{matrix}
\]}
We now proceed only with the part after the fractional point.
{\scriptsize
\[
\begin{matrix}
 \ldots &(k+1)z+k & (k+1)a-(k+1)z-(k+1) & k+1\\ 
 & k & -ka & -ka+kz & -k\\
\hline
\ldots &(k+1)z+2k & a-(k+1)z-(k+1) & -ka+kz+(k+1) & -k\\                                        
 &  & -k & ka & ka-kz & k\\
\hline
\ldots &(k+1)z+2k & a-(k+1)z-(2k+1) & kz+(k+1) & ka-kz-k & k\\
\end{matrix}
\]}
We can continue even more.
\[
\begin{matrix}
 \ldots & kz+(k+1) & ka-kz-k & k\\ 
 & -(k-1) & -(k-1)a & -(k-1)a+(k-1)z & -(k-1)\\
\hline
 \ldots & kz+2k & a-kz-k & -(k-1)a+(k-1)z+k & -(k-1)\\  
\end{matrix}
\]
As we can see, the ends of the derived representations of $x_k+x_k$ and $x_{k+1}+x_{k+1}$ consist of the same sequence of digits. If we modify this representation of $x_k+x_k$ to get the form given in the statement of the lemma, we can also change the representation of $x_{k+1}+x_{k+1}$ in a similar way. Together with the first part of the proof, it implies that the beginnings and the ends of the representations of $x_k+x_k$ and $x_{k+1}+x_{k+1}$ are the same, only the middle parts differ. Nevertheless, in this proof, we have also derived how this middle part for $k+1$ look like and it agrees with the assertion of the lemma.    
  
\end{proof}

\begin{prop}
Let $\beta$ be the dominant root of the polynomial 
\[
P(X)=X^3-aX^2-(a-z)X-1
\]
where $a>1$ and $z\in\N_{0}$ is such that $a-z>0$. If $a\geq kz+2k-1$ where $k\in\N_{0}$, then $L_{\oplus}(\beta)\geq 2k+1$. 
\end{prop}

\begin{proof}
The proof is based on a futher investigation of the representation given in the previous lemma. This representation is the greedy expansion if it satisfies two conditions. First of all, each digit is nonnegative. It holds in the case of $a\geq kz+2k-1$; the eventual negative digits are larger than the digit $a-kz-(2k-1)$. Next, we require that all the parts of our representation are lexicographically smaller than the infinite R\'{e}nyi expansion of 1, which is equal to $d_{\beta}^*(1)=(ab0)^{\omega}$. However, satisfying the previous condition, all the digits are less than $a$, thus we have the greedy expansion.

It remains to derive how many additional digits we get after the fractional point. We can claim that the greedy expansion of $x_{k}+x_{k}$ has two more digits in the end than the greedy expansion of $x_{k-1}+x_{k-1}$. Since the greedy expansion of $x_1+x_1$ has exactly three more digits after the fractional point, the greedy expansion of $x_{k}+x_{k}$ has $2k+1$ digits in this part. Hence $L_{\oplus}(\beta)\geq 2k+1$.     
\end{proof}

Note that if we fix $z$ and increase the coefficient $a$, then this lower bound on $L_{\oplus}(\beta)$ also raises. It differs from the situation when a base $\beta$ does not satisfy any of our useful properties. However, for some bases, this estimate is not optimal. As we will see in the following lemmas, the difference of certain $\beta$-integers can produce more digits after the fractional point.  

\begin{lemma}
Let $\beta$ be the dominant root of the polynomial 
\[
P(X)=X^3-aX^2-(a-z)X-1
\]
where $a>1$ and $z\in\N_{0}$ is such that $a-z>0$. Let $\langle x_k\rangle_{\beta}=a0a0\ldots a0a0$ and $\langle y_k\rangle_{\beta}=a0a0\ldots a0a$ be greedy expansions of $\beta$-integer $x_k$ and $y_k$ where $k$ is the number of $a$'s in these greedy expansions. Then one of the representations of $x_k-y_k$ in $\Z$ has the form
{\small
\[
(a-1)1(a-(z+2))(z+3)(a-2(z+2))(2z+5)\ldots ((k-2)z+2k-3)(a-(k-1)(z+2))
\]
\[
((k-1)z+(k-1))\bullet(ka-kz-(k-1))k.
\] } 
\end{lemma}    

\begin{proof}
The proof is again by induction on $k$. For $k=1$, we obtain the representation of the form
\[
(a-1)0\bullet(a-z)1,
\]
which satifies our conditions. Assumming the formula to hold for $k$, we will prove it for $k+1$. We move the fractional point in the representation of $x_k-y_k$ and add the number $a\beta-a$. 
{\small
\[
\begin{matrix}
 \ldots &(k-1)z+(k-1) & ka-kz-(k-1) & k & \bullet\\
& & a & -a & \bullet\\
\hline
 \ldots &(k-1)z+(k-1) & (k+1)a-kz-(k-1) & -a+k & \bullet\\
 & k & -ka & -ka+kz & \bullet & -k\\
\hline
 \ldots &(k-1)z+(2k-1) & a-kz-(k-1) & -(k+1)a+kz+k & \bullet & -k\\
\end{matrix}
\]}
We proceed with the modification of the end of the previous result.
{\small
\[
\begin{matrix}
 \ldots  & a-kz-(k-1) & -(k+1)a+kz+k & \bullet & -k\\ 
& -(k+1) & (k+1)a & \bullet & (k+1)a-(k+1)z & k+1\\
\hline
 \ldots  & a-k(z+2) & kz+k & \bullet & (k+1)a-(k+1)z-k & k+1\\   
\end{matrix}
\]}
As we can see, this new representation of $x_k-y_k$ has the desired form.  
\end{proof}

\begin{lemma}
Let $\beta$ be the dominant root of the polynomial 
\[
P(X)=X^3-aX^2-(a-z)X-1
\]
where $a>1$ and $z\in\N_{0}$ is such that $a-z>0$. Let $\langle x_k\rangle_{\beta}=a0a0\ldots a0a0$ and $\langle y_k\rangle_{\beta}=a0a0\ldots a0a$ be greedy expansions of $\beta$-integer $x_k$ and $y_k$ where $k$ is the number of $a$'s in these greedy expansions. Then one of the representations of $x_k-y_k$ in $\Z$ has the form
{\small
\[
(a-1)1(a-(z+2))(z+3)(a-2(z+2))(2z+5)\ldots ((k-2)z+2k-3)(a-(k-1)(z+2))
\]
\[
((k-1)(z+2))\bullet(a-kz-2(k-1))((k-1)(z+2))(a-(k-1)z-(2k-3))((k-2)(z+2))\ldots
\]
\[
\ldots (z+2)(a-z-1)1.
\]}
\end{lemma}  

\begin{proof}
This follows by the same method as in the proof of Lemma \ref{lemma:low--1}. Using the representation obtained in the previous lemma, we get the required form.
\end{proof}

Under some conditions, the representation from the previous lemma is the greedy expansion of the difference of $x_k-y_k$. 

\begin{prop}
Let $\beta$ be the dominant root of the polynomial 
\[
P(X)=X^3-aX^2-(a-z)X-1
\]
where $a>1$ and $z\in\N_{0}$ is such that $a-z>0$. If $a\geq kz+2k-2$ where $k\in\N_{0}$, then $L_{\oplus}(\beta)\geq 2k$.
\end{prop}

\begin{proof}
In this proof, we ask under which conditions the representation from the previous lemma is the greedy expansion in base $\beta$. 
Firstly, all the digits of this representation are nonnegative if $a\geq kz+2k-2$. Except for the case when $k=1$, this condition also implies that all the digits are less than $a$, thus we have the greedy expansion of $x_k-y_k$. If $k=1$, our representation is equal to $(a-1)0\bullet(a-z)1$. If $z\neq 0$, we have obtained the greedy expansion of $x_1-y_1$. For $z=0$, consider the infinite R\'{e}nyi expansion of 1, which is equal to $d_{\beta}^*(1)=(aa0)^{\omega}$. Since $a>1$, our representation is also the greedy expansion.

It is easily seen that the greedy expansion of $x_{k}-y_k$ has two more digits after the fractinal point than the greedy expansion of $x_{k-1}-y_{k-1}$, and we have exactly two such digits for $k=1$. Hence $L_{\oplus}(\beta)\geq 2k$, which completes the proof.    
\end{proof}

In some cases, this proposition improves our first lower bound on $L_{\oplus}(\beta)$. 

\begin{prop}
Let $\beta$ be the dominant root of the polynomial 
\[
P(X)=X^3-aX^2-(a-z)X-1
\]
where $a>1$ and $z\in\N_{0}$ is such that $a-z>0$. Let $k_1\in\N$ be the largest number such that $a\geq k_1z+2k_1-1$. Let $k_2\in\N$ be the largest number such that $a\geq k_2z+2k_2-2$. Then
\[
L_{\oplus}(\beta)\geq \max\{2k_1+1,2k_2\}.
\] 
\end{prop}

If $z=0$, then we can easily check that $L_{\oplus}(\beta)\geq a+2$. For $z=2$ and low values of the coefficient $a$, we can see our situation in Table \ref{tab:5}.% \ref{tab:z=2}.

\begin{table}[h] 
\centering
\begin{tabular}{|c|c|c|c|}
\hline
$a$ & $1^{st}$ lower bound & $2^{nd}$ lower bound & $L_{\oplus}(\beta)$\\
\hline
\hline 
$3$ & $3$ & $2$ & $\geq 3$\\
\hline
$4$ & $3$ & $2$ & $\geq 3$\\
\hline
$5$ & $3$ & $2$ & $\geq 3$\\
\hline
$6$ & $3$ & $4$ & $\geq 4$\\
\hline
$7$ & $5$ & $4$ & $\geq 5$\\
\hline
$8$ & $5$ & $4$ & $\geq 5$\\
\hline
$9$ & $5$ & $4$ & $\geq 5$\\
\hline
$10$ & $5$ & $6$ & $\geq 6$\\
\hline
$11$ & $7$ & $6$ & $\geq 7$\\
\hline
$12$ & $7$ & $6$ & $\geq 7$\\
\hline
\end{tabular}
\caption{Lower bounds for $z=2$} \label{tab:5}
\end{table}

In \cite{AmMaPe2}, the authors have found better lower bound on $L_{\oplus}(\beta)$ for $b=1$. Neverthless, it would be interesting to examine how precise is our bound for large values of the coefficient $b$. In the rest of the paper, we will derive an upper bound for several bases of this type. In some cases, we will obtain the exact value of $L_{\oplus}(\beta)$.

Let $\beta'$ and $\beta''$ be the conjugates of the base $\beta$. To each of them we can assign the corresponding $\Q$-isomorphism denoted by $'$ or $''$, respectively. We consequently set $\sigma(x)=(x',x'')$ for all $x\in\R$. Then the closure of the image of $\Z_{\beta}^+$, i.e., $\overline{\sigma(\Z_{\beta}^+)}$, is the Rauzy fractal of $\beta$. %In this section, we identify $'$ with the larger of our conjugates which has a smaller absolute value.
{\small
\begin{table} \label{tab:6}
\centering
\begin{tabular}{|c|c|c|c|c|}
\hline
$\text{Polynomial}$ & $3H(a,b)$ & $\sqrt{x'^2+x''^2}$ & $\langle x \rangle_{\beta}$ & $L_{\oplus}(\beta)$\\
\hline
\hline
$X^3-5X^2-5X-1$ & $28.2984$ & $32.4341$ & $0\bullet 05352241$ & $7$\\
\hline
$X^3-6X^2-6X-1$ & $38.248$ & $39.9482$ & $0\bullet 612043251$ & $8$\\
\hline
$X^3-7X^2-7X-1$ & $49.6577$ & $56.4726$ & $0\bullet 0457244261$ & $9$\\
\hline
$X^3-7X^2-6X-1$ & $37.6329$ & $53.1601$ & $0\bullet05273361$ & $6\leq L_{\oplus}(\beta)\leq7$\\
\hline
$X^3-8X^2-7X-1$ & $46.6466$ & $55.0153$ & $0\bullet07573361$ & $7$\\
\hline
$X^3-9X^2-8X-1$ & $56.6167$ & $64.2425$ & $0\bullet813064371$ & $7\leq L_{\oplus}(\beta)\leq8$\\
\hline
$X^3-8X^2-6X-1$ & $39.4296$ & $49.8007$ & $0\bullet8160451$ & $5\leq L_{\oplus}(\beta)\leq6$\\
\hline
$X^3-9X^2-7X-1$ & $47.2792$ & $75.0428$ & $0\bullet08292461$ & $5\leq L_{\oplus}(\beta)\leq7$\\
\hline
$X^3-10X^2-8X-1$ & $55.8281$ & $76.2425$ & $0\bullet0(10)593471$ & $6\leq L_{\oplus}(\beta)\leq7$\\ 
\hline
$X^3-11X^2-9X-1$ & $65.0857$ & $77.8837$ & $0\bullet09794481$ & $7$\\
\hline
$X^3-10X^2-7X-1$ & $49.3429$ & $77.7175$ & $0\bullet(10)060561$ & $5\leq L_{\oplus}(\beta)\leq6$\\
\hline
$X^3-11X^2-8X-1$ & $57.0579$ & $97.6844$ & $0\bullet0(10)2(11)2571$ & $5\leq L_{\oplus}(\beta)\leq7$\\
\hline
$X^3-12X^2-9X-1$ & $65.3281$ & $108.5816$ & $0\bullet174(11)3581$ & $5\leq L_{\oplus}(\beta)\leq7$\\
\hline    
\end{tabular}
\caption{Upper bounds for bases with negative conjugates, $x$ as the considered minimum}
\end{table} }

Our next step is to find the $\beta$-integer with the highest value of the norm
\[
\|\sigma(x)\|_2=\sqrt{x'^2+x''^2}.
\]
%which we denoted by $H(a,b)$ referring to the coefficients of the minimal polynomial of $\beta$. 
Considering the symmetry of $\Z_{\beta}$, we can restrict to nonnegative $\beta$-integers. In the case of two negative conjugates, we know the upper bound on the value of $\|\sigma(x)\|_2$ for $x\in\Z_{\beta}$, which is
\[
\sqrt{\left(\frac{a}{1-\beta'^2}\right)^2+\left(\frac{a}{1-\beta''^2}\right)^2}.
\]
Let us denote it by $H(a,b)$ referring to the coefficients of the minimal polynomial of $\beta$.
Let $w$ be some number obtained when we add or subtract two nonnegative $\beta$-integers. Obviously, $0\leq\|\sigma(w)\|_2\leq 2H(a,b)$. 

Now consider the numbers whose greedy expansion has exactly $l$ digits after the frational point, i.e., $a_{-l}\neq 0$ and $a_i=0$ for all $i< -l$. If we prove that for all such numbers it holds that their norm is greater than $2H(a,b)$, we also show that the numbers with more digits after the fractional point satisfy the same inequality. This conclusion can be deduced from the following facts. The elements which has exactly $l+1$ after the fractional point in their greedy expansions can be obtained from the numbers with $l$ such digits multiplying by $\frac{1}{\beta}$. It also implies that the components of their image under the map $\sigma$ are consequently multiplied by $\frac{1}{\beta'}$ or $\frac{1}{\beta''}$, respectively. However, $\frac{1}{\beta'},\frac{1}{\beta''}<-1$, which leads to an increase in the value of the norm of $\sigma$.     

%It follows from the fact that the set with $l+1$ digits can be examined as the set with $l$ digits and consequently the components in the image of the map $\sigma$ are multiplied by $\frac{1}{\beta'}<-1$ and $\frac{1}{\beta''}<-1$, which leads to the increase in the value of the norm of $\sigma$.  

Therefore, if we prove that all the numbers $w_l$ with $l$ digits after the fractional point satisfy $\|\sigma(w_l)\|_2>2H(a,b)$, we obtain an upper bound on $L_{\oplus}(\beta)$. Nevertheless, we have infinitely many such numbers. However, we can consider the numbers whose greedy expansion has only $l$ digits after the fractional point, no nonzero digits before, and find the minimum of the norm of these numbers. The set of such numbers is finite and can be examined by a program. Nevertheless, we must also increase the value of our bound from $2H(a,b)$ to $3H(a,b)$ to deal with this simplification. The Table \ref{tab:6} shows our results for several cubic Pisot unit bases with two negative conjugates, where we combine the computer calculations with the lower bound provided by the beginning of this section. The number $x$ denotes the found minimum.

As we can see, in some cases, we have obtained the exact value of $L_{\oplus}(\beta)$.

\subsection*{Acknowledgements}
The author is greatly indebted to Edita Pelantov\'{a} and Zuzana Mas\'{a}kov\'{a} for their advice and support.


\bigskip


\begin{thebibliography}{100}

\bibitem{Aki1} S. Akiyama, \textit{Pisot numbers and greedy algorithm},  Number Theory (1998), 9--21.

\bibitem{Aki3} S. Akiyama, \textit{Cubic Pisot units with finite beta expansions}, In F. Halter-Koch and R. F. Tichy, editors, Algebraic Number Theory and Diophantine Analysis, de Gruyter (2000), 11--26.


\bibitem{Aki2} S. Akiyama, \textit{Positive finiteness of number systems}, Developments in Mathematics 15 (2006), 1--10.

\bibitem{AmMaPe1} P. Ambro\v{z}, Ch. Frougny, Z. Mas\'{a}kov\'a and E. Pelantov\'a,  \textit{Arithmetics on number systems with irrational bases}, ull. Belg. Math. Soc. Simon Stevin 10 (2003), 641--659.

\bibitem{AmMaPe2} P. Ambro\v{z}, Z. Mas\'{a}kov\'a and E. Pelantov\'a, \textit{Addition and multiplication of beta-expansions in generalized Tribonacci base}, Discrete Math. Theor. Comput. Sci. 9 (2007), 73--88.

\bibitem{BaPeTu} L. Balkov\'a, E. Pelantov\'a and O. Turek, \textit{Combinatorial and arithmetical properties of infinite words associated with quadratic non-simple Parry numbers}, RAIRO Theor. Inf. Appl. 41 (2007), 307--328.

\bibitem{Bas} F. Bassino, \textit{Beta-expansions for cubic Pisot numbers}, in LATIN'02, Lecture Notes in Comput. Sci. 2286, Springer (2002), 141--152.

\bibitem{Ber} J. Bernat, \textit{Computation of L+ for several cubic Pisot numbers}, Discrete Math. Theor. Comput. Sci. 9 (2007), 175--194.

\bibitem{Bur} \v{C}. Burd\'{i}k, Ch. Frougny, J. P. Gazeau and R. Krejcar, \textit{Beta-Integers as Natural Couting Systems for Quasicrystals}, J. Phys. A: Math. Gen. 31 (1998), 6449--6472.

\bibitem{FroGaKre} Ch. Frougny, J. P. Gazeau and R. Krejcar, \textit{Additive and multiplicative properties of point
sets based on beta-integers}, Theor. Comp. Sci. 303 (2003), 491--516.

\bibitem{FrouSo} Ch. Frougny and B. Solomyak, \textit{Finite beta-expansions}, Ergod. Th. \& Dynam. Sys. 12 (1992), 713--723.

\bibitem{lxkva} R. Ghorbel, M. Hbaib and S. Zouari, \textit{Arithmetics on beta-expansions with Pisot bases over $F_q((x^{-1}))$}, Bull. Belg. Math. Soc. Simon Stevin 21 (2014), 241--251. 

\bibitem{GuMaPe} L. S. Guimond, Z. Mas\'{a}kov\'a and E. Pelantov\'a,  \textit{Arithmetics of beta-expansions}, Acta Arith. 112 (2004), 23--40.

\bibitem{MaTi} Z. Mas\'{a}kov\'a and M. Tinkov\'a,  \textit{Finiteness in real cubic fields}, Acta Math. Hungar. 153 (2) (2017), 318--333.


\bibitem{Mess} A. Messaoudi, \textit{Tribonacci multiplication}, Appl. Math. Lett. 15 (2002), 981--985. 

\bibitem{Par} W. Parry, \textit{On the $\beta$-expansions of real numbers}, Acta Math. Acad. Sci. Hung. 11 (1960), 401--416. 


\bibitem{Ren} A. R\'{e}nyi, \textit{Representations for real numbers and their ergodic properties}, Acta Math. Acad. Sci. Hungary 8 (1957), 477--493.

\bibitem{Salem} R. Salem, \textit{Algebraic numbers and Fourier analysis}, Heath mathematical monographs, Heath, 1963.

\end{thebibliography}
\end{document}